\documentclass[10pt]{amsart}
\usepackage{geometry} 
\usepackage{amssymb}
\usepackage{amscd}
\usepackage{amsfonts}
\usepackage{amsmath}
\usepackage{amssymb}
\usepackage[american, english]{babel}
\usepackage{bbm}
\usepackage{bookmark}
\usepackage{cmap}
\usepackage{dsfont}
\usepackage{enumerate}
\usepackage{epigraph}
\usepackage{euscript}
\usepackage[myheadings]{fullpage}
\usepackage{graphicx}
\usepackage{mathabx}
\usepackage{mathrsfs}
\usepackage{mathtools}
\usepackage{refcount}
\usepackage{stmaryrd}
\usepackage{thinsp}
\usepackage[safe]{tipa}
\usepackage{verbatim}
\usepackage{xr-hyper}
\usepackage{hyperref}
\usepackage[all]{xy}
\usepackage{url}
\usepackage{geometry} 
\usepackage{hyperref}
\usepackage{amsmath}
\newtheorem{defi}{Definition}[subsection]
\newtheorem*{defi*}{Definition}
\newtheorem{rmk}[defi]{Remark}
\newtheorem*{rmk*}{Remark}
\newtheorem*{prop*}{Proposition}
\newtheorem{thm}[defi]{Theorem}
\newtheorem*{thm*}{Theorem}
\newtheorem{cor}[defi]{Corollary}
\newtheorem*{lmm*}{Lemma}
\newtheorem{lmm}[defi]{Lemma}

\DeclareSymbolFont{largesymbols}{OMX}{yhex}{m}{n}
\DeclareMathAccent{\wideparen}{\mathord}{largesymbols}{"F3}

\title{On the D-module of an isolated singularity}
\author{Thomas Bitoun}
\email{Thomas.Bitoun@ucalgary.ca}

\begin{document}

\newtheorem{MainThm}{Theorem} 
\renewcommand{\theMainThm}{\Alph{MainThm}}. 

\maketitle

\begin{abstract} 

Let $Z$ be the germ of a complex hypersurface isolated singularity of equation $f,$ with $Z$ at least of dimension $2.$ We consider the family of analytic $D$-modules generated by the powers of $1/f$ and describe it in terms of the pole order filtration on the de Rham cohomology of the complement of $\{f=0\}$ in the neighbourhood of the singularity.
 
\end{abstract}

\section{Introduction}

The $\mathcal{D}$-modules generated by powers of a polynomial (or analytic function) $f$ have been the topic of several noted publications in the last decade, for example, \cite{MR3867305}, \cite{MR4322001}, \cite{10.1093/imrn/rnac369} and \cite{saito2022length}. On the one hand, they are elementary objects accessible to beginners in $\mathcal{D}$-module theory. On the other hand, they relate to analytic invariants and Hodge theory in deep and subtle ways.

This note provides a new, elementary approach to describing these $\mathcal{D}$-modules in the general isolated singularity case in terms of the pole order filtration on the de Rham cohomology of the algebraic link of the singularity.

Our results include: \begin{itemize} 
\item A new approach to, and a new proof of, Vilonen's characterization of the intersection cohomology $\mathcal{D}$-module (\cite[Theorem]{MR796193}), presented in Theorem \ref{thm: r} and Remark \ref{rmk: Vilonen}.

\item A new computation of the length of the $\mathcal{D}$-module of meromorphic functions (Theorem \ref{thm: r}).

\item An elementary description of the Hodge structure of the  $\mathcal{D}$-module of meromorphic functions (Theorem \ref{thm: r}).

\item A full solution to the question of the length of the $\mathcal{D}$-module generated by $1/f$ and of the corresponding Poisson cohomology (see \cite{MR3867305} and \cite[Conjecture 3.8]{MR3283930}) in Corollary \ref{cor: A}.

\item Connections between top-forms decompositions with prescribed pole order and the $\mathcal{D}$-modules generated by a power of $\frac{1}{f}$ via generalizations of Vilonen's theorem, described in Corollary \ref{cor: P}.

\item Demonstrating the importance of $\mathcal{D}$-submodules generated by pieces of the Hodge or pole order filtrations, first considered in \cite{10.1093/imrn/rnac369} and studied in Corollary \ref{cor: Hodge} and Theorem \ref{thm: M}.

\item Explaining the failure of the conjecture from \cite{MR3867305} (equivalent in terms of the Hamiltonian flow to \cite[Conjecture 3.8]{MR3283930}), as first noted in \cite{10.1093/imrn/rnac369} (see also \cite{saito2022length}). This is discussed at the end of the introduction.

\end{itemize}

Finally, we note that our approach has already led to new results, see for example the updates to  \cite{saito2022length}.
  
We now describe the contents in more technical detail. Let $f$ be a complex analytic function in $n$ variables, $n\geq3$, and assume that $Z:=\{f=0\}$ is reduced and has an isolated singularity at $o.$ Our main tool is the product pairing in the neighbourhood of the singularity between the meromorphic functions with poles along $Z$ and the top regular forms, with values in the $n$-th de Rham cohomology group $H'$ of the complement of $Z.$

Let $\delta_o$ be the irreducible $\mathcal{D}$-module supported at $o.$ In Theorem \ref{thm: r}, we show that the pairing can be interpreted as a $\mathcal{D}$-module map $r:O(*Z)_o\to \delta_o\otimes H'$, which is surjective with kernel equal to the Intersection Homology $\mathcal{D}$-module $\mathcal{L}_o.$ The latter relies on Vilonen's characterization of that $\mathcal{D}$-module, of which our result can be viewed as a generalization (e.g. Corollary \ref{cor: P}, see also Remark \ref{rmk: Vilonen}). The morphism $r$ is especially convenient to study the $\mathcal{D}$-submodules of $O(*Z)_o$ generated by powers of $1/f$ or by the pieces of Hodge filtration, see Corollaries \ref{cor: Hodge} and \ref{cor: P}.

Using the above, \cite{10.1093/imrn/rnac369} implies that the dimension of the first piece $F_0H’$ of the Hodge filtration is the reduced genus $g$ of \cite{MR3867305}, while the length of $\frac{\mathcal{D}\frac{1}{f}}{\mathcal{L}_o}$ is dim $P_0H'$, where $P_0H’$ is the set of classes generated by forms with pole order at most $1$ along $Z$. However it is well-known that the pole order filtration is, in general, strictly greater than the former, see, e.g. \cite[5.4 ii]{10.1007/BFb0086378} and \cite[(b) of Theorem 0.3]{CM_1991}. This explains the failure of \cite[Conjecture 1.7]{MR3867305}. Finally, let us note that even though the natural language of the note is that of analytic D-modules, we deduce results on length in the algebraic case as well, see Corollary \ref{cor: A}.

\subsection*{Acknowledgements} I am grateful to M. Mustata for sharing his then-preprint  \cite{10.1093/imrn/rnac369}  with S. Olano and especially to C. Sabbah for his continuous feedback highlighting in particular the importance of the analytic topology in Vilonen’s theorem. I thank also A. Dimca, for sending me some useful references. Finally, I am grateful to the anonymous referees for remarks that have helped improve the quality of the presentation. This work was supported by the Natural Sciences and Engineering Research Council of Canada (NSERC), [RGPIN-2020-06075]. Cette recherche a été financée par le Conseil de recherches en sciences naturelles et en génie du Canada (CRSNG), [RGPIN-2020-06075].

\section{A morphism of $\mathcal{D}$-modules}

\subsection{Setup and conventions} \label{setup} Let $n\geq3.$ Let $X$ be a complex analytic manifold of dimension $n,$ let $Z$ be a hypersurface of $X$ that has isolated singularities and let $U\subset X$ be the open complement of $Z.$ By restricting to an open neighbourhood of a singularity we may assume that $Z$ has a unique singularity $o,$ which we do from now on. By a $D$-module, we mean a left coherent, analytic $D_X$-module, and by a $D_o$-module, we mean a finitely generated left module over the stalk of $D_X$ at the point $o.$ For a holonomic $D$-module $M,$ we let $DR^l(M)$ be the $l$-th cohomology group of the De Rham complex of $M.$ We use the same notation $DR^l(N)$ for $N$ a holonomic $D_o$-module. We let $\mathcal{O}(\star Z)$ be the $D$-module of meromorphic functions on $X$ with poles along $Z.$ 

\subsection{Construction}

Let us first recall standard facts.

Let $k$ be a field, let $B$ be an arbitrary $k$-algebra and let $M$ be a right (resp. left) $B$-module. Then for an arbitrary $k$-vector space $W,$ the space of $k$-linear maps $L(M, W)$ from $M$ to $W$ is a left (resp. right) $B$-module, where the action is given by $bf(m):= f(mb)$ (resp. $fb(m):= f(bm)$), for all $b\in B, m\in M, f\in L(M,W).$  In the following lemma, we apply this to the right $D_0$-module $\Omega_o^n$ for $k=\mathbb{C}.$  

\begin{lmm} \label{lmm: delta} Let $V$ be a finite-dimensional complex vector space and let $o$ be a point of $X.$ For $\Omega_o^n$ the stalk at $o$ of the sheaf of differential $n$-forms, the space of linear maps $L(\Omega_o^n, V)$ from the right $D_o$-module $\Omega_o^n$ to $V$ is naturally a $D_o$-module. Moreover the $D_o$-submodule $L(\Omega_o^n, V)_o$ of linear maps annihilated by a power of the ideal of $o$ is isomorphic to the $D_o$-module $\delta_o\otimes_{\mathbb{C}}V,$ where $\delta_o$ is the irreducible $D_o$-module supported at $o.$   
	
\end{lmm} 

\begin{proof} Only the last part requires further proof. Under Kashiwara's equivalence \cite[Lemma 2.6.18]{MR1232191}, a holonomic left $D$-module $M$ supported at $o$ corresponds to the finite-dimensional vector space $\frac{M}{\frak{m}M},$  where $\frak{m}$ is the ideal of $o,$ and $M\simeq \delta_o\otimes_\mathbb{C} \frac{M}{\frak{m}M}.$ Therefore, the $D$-module $L(\Omega_o^n, V)_o$ corresponds to the vector space of linear maps from $\Omega_o^n/\frak{m}\Omega_o^n\simeq \mathbb{C}$ to $V,$ which is isomorphic to $V.$ Hence the existence of an isomorphism $L(\Omega_o^n, V)_o\simeq \delta_o\otimes_{\mathbb{C}}V. $  
\end{proof}

We will use the following lemma. For an element $\lambda$ of $L(\Omega_o^n, V),$ we denote by $\mathrm{Im}(\lambda)$ the image of the corresponding linear map $\Omega_o^n\to V.$

\begin{lmm} \label{lmm: surjectivity} Let $V$ be a finite-dimensional complex vector space and let $N$ be a $D_o$-submodule of $L(\Omega_o^n, V)_o.$ Assume that for all $v\in V,$ there exists an element $\lambda_v$ of $N$ such that $v\in \mathrm{Im}(\lambda_v).$ Then $N=L(\Omega_o^n, V)_o.$ \end{lmm}

\begin{proof} By Lemma \ref{lmm: delta} and Kashiwara's equivalence \cite[Lemma 2.6.18]{MR1232191}, $N=L(\Omega_o^n, V')_o$ for a vector subspace $V'$ of $V.$ Thus if $v\in V\setminus V',$ then for all $\lambda\in N=L(\Omega_o^n, V')_o, v\not\in \mathrm{Im}(\lambda).$ This contradicts the assumption on $N.$ Hence $N=L(\Omega_o^n, V)_o.$\end{proof} 

In our setup \ref{setup}, we denote by $\mathcal{L}$ the $D$-module pre-image in $\mathcal{O}(\star Z)$ of the intersection cohomology $D$-module $\mathcal{L}_Z\subseteq \frac{\mathcal{O}(\star Z)}{\mathcal{O}}$ associated with $Z.$ We now recall Vilonen's description of the intersection homology $D$-module in terms of residues.

\begin{thm}[Vilonen] \label{thm: Vilonen} An element $s$ of the stalk $\mathcal{O}(\star Z)_o$ is in the stalk $\mathcal{L}_o$ if and only if $\forall \omega'\in \Omega_o^n, s\omega'$ is exact, i.e. $$\mathcal{L}_o= \{s\in \mathcal{O}(\star Z)_o|\text{ For all }\omega'\in \Omega_o^n, s\omega'\in d(\Omega_o^{n-1}(\star Z))\}.$$	
\end{thm} 

\begin{proof} This is a reformulation of \cite[Theorem]{MR796193}, see \cite[5.7.21]{MR1232191} for a textbook treatment.	
We include the proof below for the benefit of the reader. Let $$V:= \{s\in \mathcal{O}(\star Z)_o|\text{ For all }\omega'\in \Omega_o^n, s\omega'\in d(\Omega^{n-1}_o(\star Z))\}.$$ It follows by a special case of the argument given in the proof of Theorem \ref{thm: r} below that $V$ is a $D_o$-submodule. We want to prove that $V=\mathcal{L}_o.$ 

Let us first show that $\mathcal{L}_o\subseteq V.$ Note that for $N \subseteq \mathcal{O}(\star Z)_o$ a $D_o$-submodule, $N\subseteq V$ if and only if the image of $DR^n(N)$ in the De Rham cohomology group $DR^n(\mathcal{O}(\star Z))_o$ vanishes. 
But $DR^n(\mathcal{L}_o)= DR^n(\mathcal{L})_o=0$ by \cite[Lemma 5.7.18]{MR1232191}, hence $\mathcal{L}_o\subseteq V.$ 
Let us now show that $\mathcal{L}_o=V.$ The quotient $\frac{V}{\mathcal{L}_o}$ is supported at the singularity since it is the case for $\frac{\mathcal{O}(\star Z)_o}{\mathcal{L}_o}.$ Therefore $\frac{V}{\mathcal{L}_o}\simeq \delta_o^j$ for some $j\geq0.$ By the long exact sequence of the $DR^i$'s applied to the short exact sequence $0\to V\to \mathcal{O}(\star Z)_o\to \frac{\mathcal{O}(\star Z)_o}{V}\to 0,$ we have that the natural map $DR^n(V)\to DR^n(\mathcal{O}(\star Z)_o)$ is an injection, because $DR^{n-1}(\delta_o)=0.$ Hence $DR^n(V)=0$ by the definition of $V.$ But using the long exact sequence of the $DR^i$'s associated with the short exact sequence $0\to \mathcal{L}_o\to V\to \delta_o^j\to 0,$ we deduce from $DR^{n-1}(\delta_o)=DR^{n+1}(\mathcal{L}_o)=0$ and $DR^n(\delta_o)=\mathbb{C}$ that $\mathbb{C}^j\simeq \frac{DR^n(V)}{DR^n(\mathcal{L}_o)}.$ Since the latter vanishes, we must have $j=0$ and $\mathcal{L}_o=V.$\end{proof}

Let us now prove the main theorems of this note. Note that the de Rham cohomology $DR^n(\mathcal{O}(\star Z)_o)$ is endowed with a Hodge structure, which we denote $H'.$   

\begin{thm} \label{thm: r} Under the hypotheses \ref{setup}, the pairing $$\mathcal{O}(\star Z)_o\times \Omega_o^n \xrightarrow{B} H';$$ $$(s, \omega') \mapsto [s\omega'],$$ where $[-]$ is the cohomology class of a form, induces a surjective homomorphism of 
$D_o$-modules 
	
$$\mathcal{O}(\star Z)_o \xrightarrow{r} L(\Omega_o^n, H')_o;$$ 

$$s\mapsto B(s,-),$$ where $o$ is the singularity. This homomorphism is compatible with the Hodge filtrations, where the Hodge filtration on $L(\Omega_o^n, H')_o$ is the one induced by that of $H'$ under Kashiwara's equivalence for Hodge $D$-modules. The kernel of $r$ is the $D_o$-module $\mathcal{L}_o.$  
	
\end{thm}

\begin{proof} 
Let us first show that the map $\mathcal{O}(\star Z)_o \to L(\Omega_o^n, DR^n(\mathcal{O}(\star Z)_o)) ;$ $$s\mapsto (\omega'\mapsto [s\omega'])$$ defines a morphism of $D_o$-modules and takes its values in $L(\Omega_o^n, DR^n(\mathcal{O}(\star Z)_o))_o.$ 

It follows directly from the definitions that the map is $\mathcal{O}$-linear. Moreover, the fact that the class of an exact form in $DR^n(\mathcal{O}(\star Z)_o)$ vanishes implies that $r$ is compatible with the actions of derivations. We may restrict ourselves to verifying it for the action of the partials $(\partial_i)_i$ corresponding to coordinates $(x_i)_i.$ Let $\omega$ be a volume form in the neighbourhood of $o$ and let $\omega^{(i)}$ be an $n-1$-form such that $dx_i\wedge \omega^{(i)}=\omega.$ Then $s\omega'= sg\omega$ for some holomorphic function $g$ and $d(sg\omega^{(i)})= \partial_i(s)\omega + s\partial_i(g)\omega.$ That is $\partial_i s\mapsto (g\omega\mapsto [\partial_i(s)g\omega)]= - [s\partial_i(g)\omega)].$ Hence the map is compatible with the right $D_o$-module action on $\Omega_o^n.$ Therefore we have a morphism of $D_o$-modules $\mathcal{O}(\star Z)_o \to L(\Omega_o^n, DR^n(\mathcal{O}(\star Z)_o)).$ 
Note that by Theorem \ref{thm: Vilonen}, the kernel of this morphism is $\mathcal{L}_o.$ But $\mathcal{O}(\star Z)_o/\mathcal{L}_o$ is supported at $o,$ hence $r$ factors through $L(\Omega_o^n, DR^n(\mathcal{O}(\star Z)_o))_o.$

That $r$ is surjective follows immediately from Lemma \ref{lmm: surjectivity}. Finally, the compatibility of $r$
 with the Hodge filtrations is a direct consequence of the construction of the Hodge filtration on the de Rham complex.\end{proof}

\begin{rmk}
Letting $Z_\infty$ be the Milnor fiber at o, we note that for $H:= H^{n-1}(Z_\infty)_1$ the unipotent monodromy part of the cohomology group of the Milnor fibre and $N$ the logarithm of the unipotent part of the monodromy, we have a natural identification of mixed Hodge structures $\gamma: H'\simeq \frac{H}{NH}.$ This follows from applying $DR^n$ to the short exact sequence $0\to M_f\simeq \mathcal{L}\to M_f^{''}\simeq \frac{\mathcal{O}(\star Z)}{\mathcal{O}}\to \frac{M_f^{''}}{M_f}\to 0$ of \cite[Remarks 3.2i]{MR2567401} and using the isomorphisms \cite[3.2.5 and 3.2.4]{MR2567401}. \end{rmk}

As a direct corollary, we get the following.

\begin{cor}\label{cor: Hodge} The image by $r$ of the $\mathcal{D}$-submodule $\mathcal{D}_oF_l \mathcal{O}(\star Z)_o$ generated by the $l$-th piece of the Hodge filtration on $\mathcal{O}(\star Z)_o$ is $L(\Omega_o^n, F_lH')_o,$ where $F_lH'$ is the $l$-th piece of the Hodge filtration on $H'.$ Therefore the length of $\frac{\mathcal{D}_oF_l \mathcal{O}(\star Z)_o}{\mathcal{L}_o}$ is $\mathrm{dim}F_lH'.$
\end{cor} 

\begin{proof}
Since $r$ is a surjective morphism of Hodge $D_o$-modules by Theorem \ref{thm: r}, we have $r(F_l \mathcal{O}(\star Z)_o)= \Sigma_{i+j\leq l} G_i\delta_o\otimes F_jH',$ where $G$ is the good filtration on $\delta_o$ induced by the standard generator of $\delta_o$ and the usual good filtration of $\mathcal{D}_o,$ see e.g. \cite[1.5.3]{MR2567401}. But the submodules $\mathcal{D}_o(\Sigma_{i+j\leq l} G_i\delta_o\otimes F_jH')$ and $\mathcal{D}_o(G_0\otimes F_lH')=L(\Omega_o^n, F_lH')_o$ are equal. Indeed the $\mathcal{D}_o$-module structure on $L(\Omega_o^n, H')_o\simeq \delta_0\otimes H'$  is such that $\mathcal{D}_o$ acts only on the left factor $\delta_0,$ which is generated by $G_0\delta_0.$ The equality follows. Therefore $r(\mathcal{D}_oF_l \mathcal{O}(\star Z)_o)=\mathcal{D}_or(F_l \mathcal{O}(\star Z)_o)=\mathcal{D}_o(\Sigma_{i+j\leq l} G_i\delta_o\otimes F_jH')= L(\Omega_o^n, F_lH')_o,$ as stated.\end{proof}

\begin{rmk} \label{rmk: Vilonen} We draw the reader's attention to the fact that Theorem 2.2.4 can be thought of as providing a new proof of Vilonen's theorem. Namely, we know that $r$ is surjective. So if we accept the elementary fact that $\mathcal{O}(\star Z)_o$ is of length $1 + \mathrm{dim}H'$ (\cite[5.7.17]{MR1232191}), we must have that the kernel of $r$ is $\mathcal{L}_o.$ But the kernel of $r$ exactly matches the description in Vilonen's theorem.\end{rmk}

We now consider $D_o$-submodules generated by an $\mathcal{O}_o$-submodule of $\mathcal{O}(\star Z)_o.$ For $M$ an $\mathcal{O}_o$-submodule of $\mathcal{O}(\star Z)_o,$ we set $\Omega_o^n(M)=M\otimes_{\mathcal{O}_o} \Omega_o^n$ and let $[\Omega_o^n(M)]\subseteq DR^n(\mathcal{O}(\star Z)_o)$ be the vector subspace of the corresponding classes of forms, namely the classes that can be represented as $mw',$ for some $m\in M$ and $w'\in \Omega_o^n.$   

\begin{thm} \label{thm: M}
	Let $M$ be an $\mathcal{O}_o$-submodule of $\mathcal{O}(\star Z)_o$  
	and let $D_oM$ be the $D_o$-submodule of $\mathcal{O}(\star Z)_o$ generated by $M.$ Assume that $\mathcal{O}(\star Z)_o$ and $D_oM$ agree generically on $Z.$ Then the quotient $\frac{D_oM}{\mathcal{L}_o}$ is isomorphic to $L(\Omega_o^n, [\Omega_o^n(M)])_o.$ Moreover we have $$D_oM= \{s\in \mathcal{O}(\star Z)_o|\text{ For all }\omega'\in \Omega_o^n, s\omega'\in M\otimes_{\mathcal{O}_o} \Omega_o^n + d(\Omega_o^{n-1}(\star Z))\}.$$	
	\end{thm}

\begin{proof} Since $\mathcal{L}_o$ is the minimal extension and $D_oM$ extends $\mathcal{O}(\star Z)_o$ generically on $Z,  D_oM$ contains $\mathcal{L}_o.$ It thus makes sense to consider the quotient $\frac{D_oM}{\mathcal{L}_o}.$ We claim that $r(D_oM)= L(\Omega_o^n, [\Omega_o^n(M)])_o,$ where $r$ is the morphism from Theorem \ref{thm: r}. Indeed, we have 

$M\subseteq r^{-1}(L(\Omega_o^n, [\Omega_o^n(M)])_o).$ Therefore $D_oM \subseteq r^{-1}(L(\Omega_o^n, [\Omega_o^n(M)])_o),$ because 

$r^{-1}(L(\Omega_o^n, [\Omega_o^n(M)])_o)$ is a $D_o$-module containing $M,$ and hence $r(D_oM) \subseteq L(\Omega_o^n, [\Omega_o^n(M)])_o.$ We then note that the equality $r(D_oM) = L(\Omega_o^n, [\Omega_o^n(M)])_o$ follows immediately from Lemma \ref{lmm: surjectivity}.\end{proof}

\section{Some corollaries}\label{section: cor}

In what follows, let us consider the filtration $P$ by order of the pole on the de Rham cohomology $DR^n(\mathcal{O}(\star Z)_o)\simeq H'.$ Namely, we let $P_l H'$ be the subspace of the classes that can be represented, via the isomorphism above, by forms having a pole of order at most $l+1$ along $Z.$  

\begin{cor}\label{cor: P} Let $f$ be a local equation of $Z,$ and let $l\geq0.$ We have the following description of $D_o\frac{1}{f^{l+1}},$ the $D_o$-submodule of $\mathcal{O}(\star Z)_o$ generated by $\frac{1}{f^{l+1}}:$ $$D_o\frac{1}{f^{l+1}}= \{s\in \mathcal{O}(\star Z)_o|\text{ For all }\omega'\in \Omega_o^n, s\omega'\in \Omega^n_o((l+1)Z) + d(\Omega^{n-1}_o(\star Z))\}.$$ It follows that the $D_o$-module length of the quotient $\frac{D_o\frac{1}{f^{l+1}}}{\mathcal{L}_o}$ is $\mathrm{dim}_{\mathbb{C}} P_l H'.$ \end{cor}

\begin{proof} Apply Theorem \ref{thm: M} to $M= \frac{\mathcal{O}_o}{f^{l+1}}.$ It follows that the quotient $\frac{D_o \frac{1}{f^{l+1}}}{\mathcal{L}_o}= \frac{D_o \frac{\mathcal{O}_o}{f^{l+1}}}{\mathcal{L}_o}$ is isomorphic to $L(\Omega_o^n, [\Omega_o^n( \frac{\mathcal{O}_o}{f^{l+1}})])_o= L(\Omega_o^n, P_lH')_o.$ Since the latter is isomorphic to $\delta_o\otimes_{\mathbb{C}}P_lH'$ by Lemma \ref{lmm: delta}, the length assertion is proved. \end{proof}

Therefore we deduce the following properties, first proved in \cite[Theorems 1.1 and 1.3]{10.1093/imrn/rnac369}, from those of the pole order filtration.

\begin{cor} Recall the hypotheses \ref{setup}. 

\begin{enumerate}

\item The $D_o$-module length of the quotient $\frac{D_o\frac{1}{f^{l+1}}}{\mathcal{L}_o}$ is at least $\mathrm{dim}_{\mathbb{C}} F_l H'.$

\item If $Z$ is quasi-homogeneous, then the inequality from 1 is an equality.

\end{enumerate}

\end{cor}

\begin{proof} Since the $D_o$-module length of the quotient $\frac{D_o\frac{1}{f^{l+1}}}{\mathcal{L}_o}$ is $\mathrm{dim}_{\mathbb{C}} P_l H'$ by the corollary \ref{cor: P}, the first assertion follows from \cite[Theorem (b)]{CM_1991} and the second from \cite[Theorem (a)]{CM_1991}.\end{proof} 

\begin{rmk} We note that, conversely, any result on the length of $\frac{D_o\frac{1}{f^{l+1}}}{\mathcal{L}_o}$ transfers by Corollary \ref{cor: P} to a statement about the pole order filtration. For example, \cite[Theorem 1]{saito2022length} describes those lengths in terms of the Gauss-Manin connection (compare with \cite[Theorem (c)]{CM_1991}). Moreover, \cite[\S 5]{10.1093/imrn/rnac369} and \cite[3.2 Example I]{saito2022length} provide new examples where the Hodge filtration is strictly contained in the pole order filtration. 
\end{rmk}

We also obtain results for algebraic $D$-modules. Let $n\geq3$ and let $g$ be a complex polynomial in $n$ variables defining a reduced irreducible hypersurface $Y$ with an isolated singularity at the origin, i.e. $|Y^{sing}|=1.$ Then for all $l\geq0,$ we denote $D^{\mathrm{alg}}\frac{1}{g^{l+1}}$ the left $D^{\mathrm{alg}}$-submodule of $R[\frac{1}{g}]$ generated by $\frac{1}{g^{l+1}},$ where $R$ is the ring of complex polynomials in $n$ variables and $D^{\mathrm{alg}}$ is the $n$-th Weyl algebra $A_n(\mathbb{C}).$ We let $IC$ be the $D^{\mathrm{alg}}$-module pre-image in $R[\frac{1}{g}]$ of the intersection cohomology $D^{\mathrm{alg}}$-module $IC_Y.$

\begin{cor} \label{cor: A} The quotient $D^{\mathrm{alg}}$-module $\frac{D^{\mathrm{alg}}\frac{1}{g^{l+1}}}{IC}$ is of length $\mathrm{dim}_{\mathbb{C}} P_l H^{n}_{dR}(B\setminus{Y}),$ where $P_l$ is the pole order filtration of the De Rham cohomology of the complement of $Y$ in a small analytic ball $B$ centred at the origin.  
\end{cor}

\begin{proof} Using the notation of \ref{setup}, the analytification functor $D_{\mathbb{C}^n}\otimes_{D^{\mathrm{alg}}}-$ is an equivalence between the category of regular holonomic $D^{\mathrm{alg}}$-modules and a full subcategory of regular $D_{\mathbb{C}^n}$-modules \cite[Proposition 7.8]{MR864073}. It follows directly from the definition that the analytification of the regular holonomic $D^{\mathrm{alg}}$-module $R[1/g]$ is the sheaf of meromorphic functions $\mathcal{O}(\star Y^{\mathrm{an}}),$ and the analytification of   
	$D^{\mathrm{alg}}\frac{1}{g^{l+1}}$ is $D_{\mathbb{C}^n}\frac{1}{g^{l+1}}.$ Moreover, because the analytification is an equivalence, the minimality of $IC$ and $\mathcal{L}$ force them to correspond to each other under the analytification functor. But, as $\frac{D_{\mathbb{C}^n}\frac{1}{g^{l+1}}}{\mathcal{L}}$ is supported at the origin, its length is the same as that of its stalk at the origin. But the natural map from $H^{n}_{dR}(B\setminus{Y})$  to $H’= DR^n(\mathcal{O}(\star Y^{\mathrm{an}}))_o$ is an isomorphism for a small enough analytical ball $B$ around $o,$ and it is compatible with the pole order filtration. Therefore the assertion follows from Corollary \ref{cor: P}. \end{proof}
	
\begin{rmk} While the corollary is presented with the constraint $|Y^{sing}|=1$ for clarity, the assertion can be extended to polynomials $g$ with multiple isolated singularities. This would involve introducing a summation over all singularities.\end{rmk}
	
\bibliography{bibfilex}
\bibliographystyle{plain}

\end{document}